\documentclass[11pt]{amsart}
\usepackage{amsmath,amsfonts,latexsym,amssymb,amscd, graphicx}

\usepackage{latexsym,enumerate,epsfig,listings}
\usepackage{amsthm,amsopn,amstext}
\usepackage[latin1]{inputenc}
\usepackage{verbatim}
\usepackage{graphicx}
\usepackage{hyperref}
\usepackage[dvipsnames, usenames]{color}
\usepackage{anysize}

\renewcommand{\P}{{\mathbb P}}
\newcommand{\E}{{\mathbb E}}
\newcommand{\R}{{\mathbb R}}

\newcommand{\0}{{\mathbf 0}}

\newcommand{\xx}{{\mathbf x}}

\newcommand{\ee}{{\mathbf e}}

\newcommand{\eps}{{\varepsilon }}

\newcommand{\Var}{{\rm Var}}
\newcommand{\esp}{{ \mathbb E}}

\newcommand{\Cov}{{\rm Cov}}

\newcommand{\eqdi}{\stackrel{d}{=}}

\newtheorem{thm}{Theorem}

\newtheorem{lem}{Lemma}
\newtheorem{prop}{Proposition}
\newtheorem{rem}{Remark}

\numberwithin{equation}{section}

\begin{document}

\title[Argmax of a process: L\'evy, Gaussian and multidimensional cases.]{On the location of the maximum of a process: L\'evy, Gaussian and multidimensional cases.}

\author[Sergio I. L\'opez]{Sergio I. L\'opez*}
\author[Leandro P. R. Pimentel]{Leandro P. R. Pimentel $\dag$ }

\address[*]{Departamento  de Matem\'aticas\\ Facultad de Ciencias, UNAM\\
	C.P. 04510, Distrito Federal, M\'exico}
\address[ $\dag$ ]{Instituto de Matem\'atica\\ Universidade Federal do Rio de Janeiro\\
	Caixa Postal 68530, CEP 21941-909 Rio de Janeiro, RJ, Brasil}

\email[*]{silo@ciencias.unam.mx}
\email[ $\dag$ ]{leandro@im.ufrj.br, lprpimentel@gmail.com}

\thanks{Sergio I. L\'opez was partially funded by the UNAM grant PAPIIT-IA104216}

\date{\today}

\begin{abstract}
In this short article we show how the techniques presented in \cite{Pi} can be extended to a variety of non continuous and multivariate processes. As examples, we prove uniqueness of the location of the maximum for spectrally positive L\'evy processes, Ornstein-Uhlenbeck process, fractional Brownian Motion and the Brownian sheet among others gaussian processes.
\end{abstract}

\maketitle

\section{Introduction and results}

The question regarding uniqueness of the location of the maximum (the maximizer) reached by a functional is relevant in optimization theory, operation research, statistics and other fields. In the deterministic cases, some information can be obtained by dealing with some specific topological structure of the functional: convexity, semicontinuity, monotonicity, etc. In the random case, when one deals with a functional of paths sampled from stochastic processes, those paths seldom satisfy those topological properties and then little information can be obtained by such methods about if there is a unique maximizer or not. \\

The problem of proving absolute continuity of functionals of random processes has been addressed since long time ago in the literature. In the seminal paper of Kim and Pollard \cite{KimPol}, in the context of functional limit theorems and its applications to statistics, it is presented a rather general theorem about the uniqueness of maximizers for gaussian processes (Lemma 2.6 there). That result (in the refined version presented by Arcones \cite{Arc}) essentially states that for a separable gaussian process $\{X_t\}_{t \in T} $, if one puts $\rho(s,t):= \Var(X_t-X_s)$ and $(T,\rho)$ turns out to be a separable metric space, then no sample path will attain its maximum in two different points. \\

In \cite{Pi} Pimentel presented another criterion which shows that the uniqueness of maximizers is equivalent to the existence of some functional derivative. Besides of proving uniqueness, in the cases where one is able to explicitly computing such derivative, it is possible to obtain some covariances formulas relating the maximizer and the maximum of the process. This functional approach goes back to Davydov \cite{Dav}, as cited by Lifshits \cite{Lif} (Lemma 3), where it is proved the equivalence of the absolute continuity of some functional to the existence of the derivative of that functional in a neighborhood along some particular direction. \\

This work can be seen as a continuation of \cite{Pi}. Our contribution is to show that the techniques presented in that article are rather robust and can be adapted to more general examples. First, let us show how the assumption on continuity of the sample paths can be relaxed. \\

 Consider a real c\`{a}dl\`{a}g  function $h$ with supremum value denoted by $S(h)$ and define the set
$$ \mathcal{Z} (h) := \Big\{ z \in K :  \forall \eps>0 ,  \,  \sup_{ x \in [z-\eps, z+ \eps] } h(x) = S(h) \Big\} ,$$ 
which is the set of points which are infinitesimally close to be a maximizer (which will be called quasi-maximizers). We will say that there is a unique quasi-maximizer for $h$, if the set $\mathcal{Z}(h)$ consists of only one point, denoted by $Z$.  We derive a generalization of Theorem 1 in \cite{Pi} for c\`{a}dl\`{a}g  processes, while we also allow the use of non-linear perturbations.

\begin{thm}\label{teo:nl}
	Let $\{ X(z): z \in K\} $ be a c\`{a}dl\`{a}g  process defined on a compact set $K \subseteq \R$ such that its supremum $S$ satisfies $\E|S| < \infty$. For $a$ in $\mathbb{R}$ define
	$$ X^a(z)= X(z) +  a \rho (z) \qquad \forall \, z \in K  ,$$
	where $\rho$ is an strictly increasing continuous real function. Put $S^a =S(X^a) $ and $s(a):= \E S^a$. Then $X$ has a unique quasi-maximizer $Z$ if and only if the derivative of the map $a \mapsto s(a)$ exists at $a=0$. In that case, we have that 
	 $$\E ( \rho(Z) )  =  \frac{d}{da} s(a) \Big|_{a=0}. $$
\end{thm}

Based on Theorem \ref{teo:nl}, we give some extension to an arbitrary dimension $d$. However, we restrict ourselves to continuous processes in this case.

\begin{thm}\label{teo:multnl}
	Let $\{ X(z): z \in K\} $ be a continuous process defined on a non-empty compact set $K \subseteq \R^d$ such that its maximum $S$ satisfies $\E|S| < \infty$. For $a=(a_1, \dots a_d)$ in $\R^d$, define
	$$ X^a(z)= X(z) + \sum_{i=1}^d a_i \rho_i (z_i)  \qquad \forall \, z=(z_1,...,z_d) \in K  ,$$
	where $\rho_i$ is a strictly increasing continuous real function, for each $i=1,\dots,d$. Denote by $S^a$ the maximum of $X^a$, and set $s(a):= \E S^a$. Then the location $Z=(Z_1,...,Z_d)$ of the maximum of $X$ is unique if and only if the gradient of the map $a \mapsto s(a)$ exists at $a=\0$. In that case, we have that $$( \E ( \rho_i(Z_i) ) )_{i=1}^{d}  = \nabla s(\0). $$
\end{thm}

In Section 2, we present examples of applications of Theorems \ref{teo:nl} and \ref{teo:multnl} to spectrally positive L\'evy processes and continuous gaussian processes. In Section 3 we give the proofs of the main results and the examples presented.

\section{Examples}

In this section we give some examples of processes where Theorems \ref{teo:nl} and \ref{teo:multnl} can be applied to derive uniqueness of the location of its supremum.

\subsection{L\'evy processes}

Let us consider some spectrally positive L\'evy process $\{X(t): t \in [0,1] \}$. We have its L\'evy-It\^o representation
\begin{equation}\label{Levy-Ito}
 X(t)= c t + \sigma B(t) + Y(t) ,
\end{equation}
where $c$ is some real, $\sigma$ a positive number, $B$ is a standard Brownian Motion and $Y$ is a process conformed by the discontinuities of $X$, that is
$$ Y(t) = \sum_{ 0 \leq s \leq t} \Delta X_s .$$ 
The condition of being spectrally positive is translated on $X$ non having negative jumps: $\Delta X_t \geq 0$ for every $0 \leq t \leq 1$.

Consider a a c\`{a}dl\`{a}g  function $h$ without negative jumps. Note that, in this case, to be a quasi-maximizer is the same as to be a maximizer: the valuation of a quasi-maximizer reaches the value $S(h)$. This observation and Theorem \ref{teo:nl} implies the next result.

\begin{thm}\label{Levy}
For an spectrally positive L\'evy process $X$ with representation \eqref{Levy-Ito} it happens that:
\begin{enumerate}
	\item If $\sigma > 0$ the supremum of the process is reached at a unique point.
	\item When $\sigma =0$ but $c$ is different from $0$, we also have that the supreme is attained at a unique point.
	\item Consider the scenario where $\sigma=0$, $c=0$. Define the random variable 
	$$L:= \inf \Big\{ t \in [0,1] :  X_t = \sup_{ s \in [0,1]  } X_s  \Big\}.$$
	Then the event of the supremum being reached at a unique point (denoted by $A$) is the same as the event
	$\{ L = 1 \}$. Moreover, its probability is equal to $\P (\tau =0 )$, where $\tau$ is the exit time of process $X$ from zero.
\end{enumerate}
\end{thm}

As a consequence, the equality $ \P ( A ) = \P( \tau=0 )$ is valid in all the cases enumerated in Theorem \ref{Levy}. 

\begin{rem} In the case where $\sigma >0$ the proof can be applied to general L\'evy processes, giving the result of uniqueness for the quasi-maximizer of the process.
\end{rem}

\subsection{Gaussian processes}

As a corollary of Theorem \ref{teo:nl} we have the next result.

 \begin{thm}\label{theo:gauss1} 
 Let $G$ be a zero mean continuous gaussian process defined on $K=[0,T]$ such that $G(0)=0$, and with covariance function $R$. Let $f:K \rightarrow \mathbb R$ be a deterministic c\`{a}dl\`{a}g function, and set $ X(z)= G(z) +f(z)$, for $z \in K$.  Let $S$ be the maximum of $X$ and assume that $\esp(S^2)< \infty$. Also, assume that $R(z,T)$ is an increasing function of $z$. Then, the location $Z$ of the maximum of $X$ is $a.s.$ unique and it satisfies
$$ \E ( R(Z,T) ) = \Cov(S, X( T)) .$$
\end{thm}

To prove Theorem \ref{theo:gauss1} it is necessary to develop some method for computing the derivative of the map $s(a)$. This can be accomplished for gaussian processes by adapting the method used in \cite{Pi} for Brownian motion, while using anticipative representations of gaussian bridges (see \cite{GasSo}). The proof is presented in Section 3.\\

In the following, we show some concrete examples of Theorem \ref{theo:gauss1}.

\noindent\paragraph{\bf Ornstein-Uhlenbeck process}

Let $X$ be a centered Ornstein-Uhlenbeck process starting at the origin:
$$X_s = \frac{\sigma}{ \sqrt{2 \gamma} } e^{- \gamma s}  W \Big( e^{2 \gamma s} -1 \Big)  \qquad \forall \, 0 \leq s \leq t ,$$
where $\gamma$ and $\sigma$ are positive numbers and $W$ denotes a standard Brownian motion. Then 
$$ \Cov (X_s,X_t)= \frac{\sigma^2}{2 \gamma} e^{- \gamma t} ( e^{ \gamma s} - e^{- \gamma s} )    \qquad \, \forall 0 \leq s \leq t ,$$
so $\Cov (X_s,X_t)$ is strictly increasing on $[0,t]$, as a function of $s$. Hence we can apply Theorem \ref{theo:gauss1} to obtain uniqueness of the location of maximum of $X$ on $[0,t]$ and the formula:
$$ \E ( e^{\gamma Z} - e^{ -\gamma Z} ) = \frac{ 2\gamma}{ \sigma^2} e^{ \gamma t}  \Cov(S, X(t))  ,$$
where $S$ is the maximum of $X$ on $[0,t]$ and $Z$ its location. \\

\noindent\paragraph{\bf  Fractional Brownian Motion.}
Consider a centered gaussian process $X$ defined on $[0,t]$ starting at the origin with covariance function given by
$$ \Cov (X_u,X_v)= \frac{1}{2} ( u^{2H} + v^{2H} - |u-v|^{2H}  )  \qquad \forall \, u,v \in [0,t]\,,$$ 
where $H \in (0,1)$. By checking that the function $\rho(s)= \Cov(X_s, X_t) $ has positive derivative on $[0,t]$, one can apply Theorem \ref{theo:gauss1}  conclude uniqueness of the location of the maximum of $X$ and obtain the formula
$$ \E (z^{2H}-(t-Z)^{2H} ) = 2 \Cov (S,X(t)) - t^{2H} .$$

By dealing with orthogonal decompositions for multivariate gaussian process (see Lemma \ref{lem1gau} in Section $3$) one can apply Theorem \ref{teo:multnl} to obtain an extension of Theorem \ref{theo:gauss1} for an arbitrary dimension.

\begin{thm}\label{theo:gauss} 
	Let $G$ be a zero mean continuous gaussian process defined on a non-empty compact set $K$ cointaining the origin such that $G(\0)=0$ and with covariance function $R$. Let $f:K \rightarrow \mathbb R^d$ be a deterministic continuous function and set  $X(z)= G(z) +f(z)$, for $z \in K$.  Let $S$ be the maximum of $X$ and assume that $\esp(S^2)< \infty$. Also, assume that there exists a collection of points $\{t^i\}_{i=1}^{d}$ in $K$ such that 
	\begin{enumerate}
		\item the gaussian vector $(X(t^1),...,X(t^d))$ has an invertible and diagonal covariance matrix,
		\item for every $z$ in $K$ we have that $R(z, t^i)$ only depends on $z_i$ (the $i$-coordinate of $z$),
		\item $R(z,t^i)$ is a strictly increasing function of $z_i$, and
		\item it takes the value $0$ for $z_i=0$.
	\end{enumerate}
	 Then, the location $Z$ of the maximum of $X$ is $a.s.$ unique and it satisfies
	$$ \E ( R(Z_i, t^i ) ) = \Cov(S, X( t^i))  \qquad \forall \, 1 \leq i \leq d.$$
\end{thm}

We show some applications of this result.

\noindent\paragraph{\bf Brownian sheet with Brownian frontier}

Consider the centered gaussian process $\{ B(z) : z \in K\}$ with covariances given by
$$ \Cov ( B(u), B(v)) = \prod_{i=1}^{d} (u_i \wedge v_i), \textrm{ for } u=(u_1,...,u_d)\,\mbox{ and }\,v=(v_1,...,v_d) \textrm { in } K:=\prod_{i=1}^d[0,T_i] \subseteq \mathbb R ^d. $$
Let $\{ W^i\}_{i=1}^{d}$ be a collection of independent standard Brownian motions which are independent from $S$. Define the process $\{X(z): z \in K\}$ by
$$ X(z):= B(z) + \sum_{i=1}^{d} W^i(z_i),  \textrm{ for } z=(z_1,...,z_d) \in K.$$
We will say that $X$ is a Brownian sheet with initial conditions on the orthant's frontier given by independent Brownian motions.  Choose $t^i:=T_i \ee^i$ for $ 1 \leq i \leq d$.  By direct computing, one obtains that $R(z, T_i \ee^i)=z_i$, for $z=(z_1,...,z_d)$ in $K$. In particular, $R(T_j \ee^j, T_i \ee^i)= T_i \delta_{i,j}$. Hence we can apply Theorem \ref{theo:gauss} to obtain that the location of the maximum is $a.s.$ unique and 
$$ ( \E ( Z_i)  )_{i=1}^{d} = \Big( \Cov(S, X( T_i \ee^i)) \Big)_{i=1}^{d} .$$

\noindent\paragraph{\bf Linear covariances} Here $X$ is a centered gaussian process on $K:=[0,T_1] \times ... \times [0,T_d]$, with covariances given by
$$ \Cov(X_u,X_v)= \sum_{i=1}^{d} u_i v_i  \quad \forall u=(u_1,...,u_d), v=(v_1,...,v_d) \in K .$$
We have that $R(z, T_i \ee^i) =T_i z_i$ so we have uniqueness of the location of the maximum, by Theorem \ref{theo:gauss}.\\

\noindent\paragraph{\bf Additive Brownian Motion}
For $\xx=(1,n) \in (0,\infty)  \times \mathbb N$ denote by $\Gamma(\xx)$ 
the set of all increasing sequences $\gamma=(z_0=0 \leq z_{1} \leq \dots \leq z_{n+1}=1)$. Define the passage time of a sequence $\gamma$ by 
\begin{equation*}\label{passage} 
L(\gamma):=\sum_{i=0}^{n} B^{(i)}(z_{i+1})-B^{(i)}(z_i)\, ,
\end{equation*}
where $\{B^{(i)} \}_{i \geq 0}$ is an independent collection of standard Brownian motions. The last-passage time between $\0 = (0,0)$ and $\xx$ is defined as
\begin{equation}\label{lastpassage} 
L(\0,\xx):=\sup_{\gamma\in\Gamma(\xx)} L(\gamma) \,.
\end{equation}
The \textit{geodesic} from $\0$ to $\xx$ is defined as the path $\gamma^*$ in $\Gamma(\xx)$ such that $L(\gamma ^*) = L(\0, \xx)$. In order to be well defined it is necessary to show that the supremum \eqref{lastpassage} is attained by a unique sequence $\gamma^*$; we do it in the following. Precise asymptotic results for this model have been obtained, see \cite{HMO}. Let us note that having uniqueness of the geodesic from $\0$ to $\xx$ is equivalent to have a unique location for the supremum of the process:
$$ X( u_1,...,u_n ):= B^{(0)}(0,u_1) + B^{(1)}(u_1,u_1+u_2) +...  +B^{(n-1)} \Big( \sum_{i=1}^{n-1} u_i ,  \sum_{i=1}^{n} u_i \Big) + B^{(n)} \Big( \sum_{i=1}^{n} u_i ,  1 \Big), $$
defined for points in the compact domain
$$ K= \Big\{  u=(u_1,...,u_n) : u_i \geq 0, \, \sum_{i=1}^{n} u_i \leq 1 \Big \}. $$
The covariance $R$ of this process satisfy
$$ R( u,  \ee^j) = \Cov \Big( B^{(j-1)} \Big(\sum_{i=1}^{j-1} u_i , \sum_{i=1}^{j} u_i \Big) , B^{(j-1)} (0,1) \Big) = u_j \qquad \forall \, u=(u_1,...,u_n) \in K,$$
hence Theorem 5 can be applied to conclude uniqueness of the location of the supremum of $X$.

\section{Proofs}

\subsection{Uniqueness criteria}

We generalize Lemma $1$ in \cite{Pi}. The proof, suggested by Remark $1$ in \cite{Pi}, follows the same lines and we prove it for the sake of completeness. Let $h: K \subseteq \mathbb{R}^d \rightarrow \mathbb{R}$  be a function defined on a non-empty compact set $K$. Denote its supremum value by $S(h)$ and define the set
$$ \mathcal{Z} (h) := \Big\{ z \in K :  \forall \eps>0 ,  \,  \sup_{ x \in B_{\eps} (z) } h(x) = S(h) \Big\} .$$ 
In the following, we consider the case when $h$ is real so one can define
$$  Z_l (h) := \inf \mathcal{Z} (h) , \qquad Z_r :=\sup \mathcal{Z}(h).$$

\begin{lem}\label{lem:one}
	For $a$ fixed, define the function $h^{a} : \mathbb{R} \rightarrow \mathbb{R}$ by 
	$$ h^{a} (z) := h(z) +  a \rho(z)   \qquad \forall z \in K,$$
	where $\rho$ is a strictly increasing continuous real function and $K$ is a non-empty compact set. Then
	$$ \lim_{ a \rightarrow 0^- } Z_l ( h^{a} ) = Z_l ( h) , \quad \lim_{a \rightarrow 0^+} Z_r (h^{a} ) = Z_r ( h) .$$
	Moreover,
	$$ \lim_{a \rightarrow 0^-} \frac{ S(h^{a}) -S(h)}{a}= 
	\rho (Z_l ( h) ), \qquad \lim_{ a \rightarrow 0^+ } \frac{S(h^{a}) -S(h) }{a}= \rho (Z_r ( h) ). $$
\end{lem}

\begin{proof}
Note that for a c\`{a}dl\`{a}g function $h$ and $z$ in $\mathcal Z (h)$ it must happen that $h(z^-)=S(h)$ or $h(z^+)=S(h)$. Denote by $z^o$ to one of those directional limits, the one such that $h(z^o)=S(h)$. Similarly, we will use $(z^{a})^o$ for denoting some directional limit such that $h^a( (z^{a})^o ) =S(h^a)$, where the election of the direction depends on $a$. For $ j=l,r$ it happens that
\begin{equation}\label{eqbas}
S + a \rho(Z_j)= h(Z_j^o) + a \rho(Z_j) \leq \sup_{z \in K }  \{h(z) + a \rho (z) \} =S^a = h( (Z_j^{a})^o ) + a \rho(Z_j^{a}) \leq S + a \rho(Z_j^a),
\end{equation} 
where first inequality is due to maximality of $S^a$ and the second one to maximality of $S$, besides continuity of $\rho$. This implies
\begin{eqnarray}\label{eqsign}
0 \leq a ( \rho(Z_j^a) - \rho(Z_j) ) \qquad j=l,r,
\end{eqnarray}
and also
\begin{equation}\label{eqdif}
0 \leq a ( \rho(Z_j^a) - \rho(Z_j) ) - ( h(Z_j^o)-h((Z_j^{a})^o)  )   \qquad j=l,r.
\end{equation}
Since $h(Z_j^o) \geq h((Z_j^{a})^o)$, by Equation \eqref{eqsign} we have that
\begin{displaymath} 
\left \{ \begin{array}{cc}
\rho(Z_j^a) \leq \rho(Z_j) &  \textrm{ for } a < 0 ,\\
\rho(Z_j^a) \leq \rho(Z_j) & \textrm{ for }  a>0.
\end{array} \right. 
\end{displaymath}

Assume that the first conclusion of Lemma \ref{lem:one} does not hold, that is $\lim_{a \rightarrow 0^-} Z_l^a \neq Z_l$. Then since $\rho$ is continuous and injective we have that
$$ \lim_{a \rightarrow 0^-} \rho( Z_l^a) = \rho \Big( \lim_{a \rightarrow 0^-} Z_l^a \Big) \neq \rho( Z_l).$$
This, combined with \eqref{eqsign}, implies the existence of $\delta >0$ and some sequence of real numbers $\{a_n\}$ such that $a_n \rightarrow 0^-$ when $n$ goes to infinity, and for all $n \geq 1$ it happens that $\rho(Z_l^{a_n} )  \leq \rho(Z_l) - \delta $. Since $\rho(K)$ is compact, the sequence $\{ \rho(Z_l^{a_n}) \}$ has some convergent subsequence  $\{ \rho(Z_l^{a_{n_k}}) \}$ to some limit $y$. Let us call $x$ to the value in $K$ such that $\rho(x)=y$. Since $\rho^{-1}$ is also continuous, it happens that $\{ Z^{a_{n_k}} \}$ converges to $x$. Note that $\rho(x)  = \lim_{k \rightarrow \infty} \rho(Z_l ^{a_{n_k} } )  \leq \rho(Z_l) - \delta$, which implies that $x < Z_l$.

Recall Equation \eqref{eqdif}:
\begin{equation}\label{eqrhoh}
0 \leq  a_{n_k} [  \rho(Z_l^{a_{n_k}} ) - \rho(Z_l) ] +   h((Z_l^{a_{n_k}})^o) - h(Z_l^o),
\end{equation}
using that $\rho(Z_l^{a_n} )  \leq \rho(Z_l) - \delta$ we obtain
$$a_{n_k}  \delta \leq  a_{n_k} [ \rho(Z_l) - \rho(Z_l^{a_{n_k } }) ]  \leq h((Z_l^{a_{n_k}})^o)- h(Z_l^o).$$
This means that for every $k$, when can find some sequence $\{x_k\}$ in $K$ such that $|x_k - Z_l^{a_{n_k}} | < \frac{1}{k} $ and satisfies that 
$h(x_k) \geq a_{n_k}  \delta + h(Z_l^o) - \frac{1}{k}$. The convergence of $Z^{a_{n_k}}$ to $x$ implies the convergence of $\{x_k\}$ to $x$. Then we conclude that $x$ is in the set $\mathcal Z (h)$ and $x < Z_l$ which is a contradiction. The proof for $Z_r$ is analogue.\\

Now, we prove the second part of the lemma. From \eqref{eqbas} we have that for $j$ in  $\{l,r\}$
\begin{eqnarray*}
	S - S^a + a \rho (Z_j) &\leq& 0, \\
	S - S^a + a \rho (Z_j^a) &\geq& 0. 
\end{eqnarray*}
For the case $a <0$ and $j=l$ this implies that 
\begin{equation}\label{eqsandw}
- \infty < \rho( \inf K ) - \rho( \sup K) \leq \rho(Z_l^a) - \rho(Z_l) \leq  \frac{S^a-S}{a} - \rho (Z_l) \leq 0 .
\end{equation}
Similarly, for the case $a>0$ and $j=r$:
$$ 0 \leq \frac{S^a-S}{a} - \rho (Z_r) \leq \rho(Z_r^a) - \rho(Z_r) \leq \rho( \sup K ) - \rho( \inf K ) < \infty .$$
By taking the limit in \eqref{eqsandw} when $a \rightarrow 0^-$ we obtain that
$$ \lim_{a \rightarrow 0^-} \rho(Z_l^a) - \rho(Z_l) \leq \lim_{a \rightarrow 0^-} \frac{S^a-S}{a} - \rho (Z_l) \leq 0,$$
and since $\lim_{a \rightarrow 0 ^-} \rho(Z_l^a)=\rho (Z_l)$, by the first part of the lemma and the continuity of $\rho$, the result follows. Taking the limit when $a \rightarrow 0^+$ leads to the analogous result for $\rho(Z_r)$. 
\end{proof}

Now, we return to the case of an arbitrary dimension $d$. For our purposes we will consider $h: K\subseteq \mathbb{R}^d \rightarrow \mathbb{R}$ to be a continuous function with a non-empty compact set $K$ as a domain. Similarly to the preceeding, let
$ S(h)$ be  the supremum of the values of the function $h$ and the set of maximizers
$$ \mathcal{Z} (h) := \{ z \in K : h(z) = S(h) \} .$$ 
For $1 \leq i \leq d$ put
$$ \mathcal{Z}^i (h) := \{ x \in \mathbb{R} \, :  \, \exists z \in K \textrm{ such that }  h(z)=S, \pi _i(z) =x \} $$
where $\pi_i$ denotes the projection on the $i$-coordinate. Define also
$$  Z^i_l (h) := \inf \mathcal{Z}^i (h) , \qquad Z^i_r :=\sup \mathcal{Z}^i (h)  , \qquad \forall 1 \leq i \leq d.$$

\begin{lem}\label{lem:mult}
For $a$ real and $1 \leq i \leq d$ define the function $h^{a,i} : K \subseteq \mathbb{R}^d \rightarrow \mathbb{R}$ by 
$$ h^{a,i} (z) := h(z) +  a \rho_i(z_i)   ,$$
 for $z=(z_1,...,z_d)$, where $\rho_i$ is an increasing real function. Then
$$ \lim_{ a \rightarrow 0^- } Z^{i}_l ( h^{a,i} ) = Z^{i}_l ( h) , \quad \lim_{a \rightarrow 0^+} Z^{i}_r (h^{a,i} ) = Z^i_r ( h) .$$
Moreover,
$$ \lim_{a \rightarrow 0^-} \frac{ S(h^{a,i}) -S(h)}{a}= 
\rho_i (Z^i_l ( h) ), \qquad \lim_{ a \rightarrow 0^+ } \frac{S(h^{a,i}) -S(h) }{a}= \rho_i (Z^i_r ( h) ). $$
\end{lem}

\begin{proof}

Fix $1 \leq i \leq d$ and set 
$$E_x^i:=  \{ z \in K : \pi_i(z) = x \}.$$
For those real numbers $x$ such that $E_x^i$ is non empty define the function
$$ f_i(x) :=  \max_{ z \in E_x^i } h(z) .$$
We will verify that $f_i$ satisfies the hypothesis to apply Lemma \ref{lem:one} and the result will follow from it.

\textbf{Claim:}
\begin{enumerate}
 \item For all $x$, $E_x^i$ is compact.
 \item The set $\pi_i(K)$ is a non-empty compact set.
 \item The function $f_i : \pi_i (K) \rightarrow \mathbb{R}$ is continuous.
\end{enumerate}

The first part of the Claim is immediate, since $E_x^i = \pi_i ^{-1} (\{x\})  \cap K$ is the intersection of two compact sets. The second part follows from $\pi_i(K)$ being the image under the continuous function $\pi_i$ of the non-empty compact set $K$. Let $\eps> 0$ and $x$ in  $\pi_i(K)$ fixed. We now prove that there exists $\delta>0$ such that for $|x-y| < \delta$ it happens that $|f_i(x) - f_i(y)| < \eps$.

Since $E_x^i$ is a non-empty compact set and $h$ is continuous, there exists $z_x \in E_x^i$ such that
$$ h(z_x) = \max_{z \in E_x^i} h(z) .$$
By the continuity of $h$, there exists $\delta >0$ such that for any $ u $ with 
$  || u - z || < \delta$ it happens $ |h(u) - h(z) | < \varepsilon$, where $|| (u_1,...,u_n)||:= \max |u_i |$. 

Let $y$ in $\pi_i(K)$ be such that $ |y-x| < \delta$. Then, there exists $z_y \in K$ such that $h(z_y) = \max_{ z \in E_y^i} h(z)$. 
Suppose now that $ h(z_y) \leq h(z_x)$.  We define $z^*_y$ to be identically to the vector $z_x$ but in the $i$-coordinate has the value $y$. Since $h(z^y)$ is the maximum of the values of $h$ on the set $E_y^i$ we have
\begin{equation}\label{eqsandh}
 h(z^*_y) \leq h( z_y) \leq h(z_x).
\end{equation}
On the other hand
$$ || z^*_y - z^x || = |y - x| < \delta$$
and then $|h( z^*_y) - h(z^x)| < \varepsilon$. From this and \eqref{eqsandh} we conclude that $|f_i(x)- f_i(y)|= |h(z_x)-h(z_y)| < \varepsilon$. In the case that $ h(z_x) \leq h(z_y)$ we define $z^*_x$ to be identical to $z_y$ but in the $i$-coordinate has the value $x$. The proof is symmetrical so the Claim is proved.\\

Note  that $\mathcal{Z}^i (h)$ can be rewritten as 
$$ \mathcal{Z}^i (h) = \{ x \in \pi_i (K) : f_i(x) = S( f_i(x) ) \} = \mathcal{Z} (f_i), $$
and then $ Z_l (f_i) = \inf \mathcal{Z}^i (h)$ and  $Z_r (f_i) = \sup \mathcal{Z}^i (h)$. By the Claim, we can apply Lemma \ref{lem:one} to $f_i$ and we obtain
$$ \lim_{ a \rightarrow 0^- } Z^i_l ( f_i(x) + a \rho_i (x) ) = Z^i_l ( f_i(x) )  \quad  \lim_{ a \rightarrow 0^+ } Z^i_r ( f_i(x) + a \rho_i (x) ) = Z^i_r ( f_i(x) )   $$
and, 
\begin{eqnarray*}
 \lim_{a \rightarrow 0^-} \frac{ S(f_i(x) + a \rho_i(x)) -S(f_i(x)) }{a}  &=& 
 \rho_i (Z^i_l ( f_i(x) ) ),   \\ 
 \lim_{a \rightarrow 0^+} \frac{ S(f_i(x) + a \rho_i(x)) -S(f_i(x)) }{a}   &=& 
\rho_i (Z^i_r ( f_i(x) ) ).
\end{eqnarray*}
But
\begin{eqnarray*}
S(f_i(x) +a \rho_i(x) ) &=& \max_{x \in \pi_i(K) }  \Big\{  \max_{z \in E_x^i}  \{ h(z) \} + a \rho_i(x)   \Big\}  \\
&=& \max_{x \in \pi_i(K) }  \Big\{  \max_{z \in E_x^i} \{ h(z) + a \rho_i(x) \}  \Big\} \\
&=& \max_{z \in K} \{ h(z) + a \rho_i( \pi_i(z) ) \}    = S(h^{a,i}),
\end{eqnarray*}
and similarly $Z^i_l ( f_i(x) + a \rho_i (x) )=Z_l(h^{i,a})$, $Z^i_r ( f_i(x) + a \rho_i(x) )=Z_r(h^{a,i})$, so the result follows.
\end{proof}

\textbf{Proof of Theorem \ref{teo:nl} and Theorem \ref{teo:multnl}:} 
\begin{proof}
The proof follows the same lines of Theorem $1$ of \cite{Pi}. We write explicitly the proof of Theorem \ref{teo:multnl} to show the details in the multivariate setting. \\

We have that 
\begin{equation}\label{eqdom}
 \E|S^a| \leq \E|S| + \sum_{i=1}^{d} |a_i|  \max_{x \in \pi_i(K) } | \rho_i(x) | < \infty .
\end{equation}
If the gradient of $s$ exists at the origin, for all $1 \leq i \leq d$ we have that
$$ \frac{ \partial  s}{\partial a_i} \Big| _{a=\0}= \lim_{a_i \rightarrow 0^-} \frac{ \E ( S(h+a_i \rho_i) - S(h) ) }{a_i} = \lim_{a_i \rightarrow 0^+} \frac{ \E ( S(h+a_i \rho_i) - S(h) ) }{a_i} ,$$
and by \eqref{eqdom} we can use dominated convergence to conclude that
\begin{equation}\label{eqexp}
 \E \Big( \lim_{a_i \rightarrow 0^-} \frac{ S(h+a_i \rho_i) - S(h) }{a_i}  \Big)  = \E  \Big( \lim_{a_i \rightarrow 0^+} \frac{ S(h+a_i \rho_i) - S(h) }{a_i}  \Big) .
\end{equation}
From \eqref{eqexp} and the first part of Lemma \ref{lem:mult} we deduce that $\E( \rho_i(Z_l^i)) = \E( \rho_i (Z_r^i))$. Since $Z_l^i \leq Z_r^i$ and $\rho_i$ is increasing, we have that $\rho_i(Z_l^i) \leq \rho_i(Z_r^i)$. This fact combined with \eqref{eqexp}, implies that $\rho_i(Z_l^i) = \rho(Z_r^i)$ almost surely. Since $\rho_i$ is injective, we conclude that $Z_l^i = Z_r^i$ almost surely, for each coordinate $i$, and this implies the $a.s.$ uniqueness of the location $Z$. Note that in this case $\E( \rho_i (Z_i) ) =  \frac{ \partial  s}{\partial a_i} \Big| _{a=\0} $. \\

For the converse, assume that the location of the maximum is unique. For each $i$, we have that $\rho_i(Z_l^i) = \rho_i(Z_r^i)$, and by Lemma \ref{lem:mult} we have that
$$ \lim_{a_i \rightarrow 0^-} \frac{ S(h+a_i \rho_i) - S(h) }{a_i}   = \lim_{a_i \rightarrow 0^+} \frac{ S(h+a_i \rho_i) - S(h) }{a_i} .$$
The dominated convergence theorem implies that
$$ \frac{ \partial  s}{\partial a_i} \Big| _{a=\0}= \lim_{a_i \rightarrow 0^-} \frac{ \E ( S(h+a_i \rho_i) - S(h) ) }{a_i} = \lim_{a_i \rightarrow 0^+} \frac{ \E ( S(h+a_i \rho_i) - S(h) ) }{a_i} ,$$
which proves that the gradient of $s$ exists at $a=\0$.
\end{proof}

\subsection{L\'evy processes}

\textbf{Proof of Theorem \ref{Levy}:} 
\begin{proof}
The case where $Y(t) \equiv 0$ corresponds to Theorem $2$ in \cite{Pi}.  Consider the scenario where $\sigma>0$. Let $A$ be the event where $X$ attains its maximum in a unique position. Then
$$ \mathbb{E} ( 1_A) = \mathbb{E} (  \mathbb{E} ( 1_A | Y(t) =f(t)  )   ) , $$
where $f(t)$ is a function varying overall c\`{a}dl\`{a}g  functions in $[0,1]$. So, by proving uniqueness of the process $ct+ \sigma  B_t + f(t)$ for any c\`{a}dl\`{a}g  function $f$, we would be done for this case. But this follows directly from Lemma $2$ in \cite{Pi} and Theorem \ref{teo:nl} (in the same way as Theorem $2$ is proven in \cite{Pi}, the only difference here is that $f$ could be discontinuous). \\

Assume from now on that $\sigma=0$. Directly from the definition, one could check that  if $c>0$, for any $a$ such that $c+a>0$, it happens that $S^a=S+a$ so the quotient $\frac{S^a-S}{a}$ is equal to $1$, for all $a$ sufficiently small. Similarly, for $c<0$, it happens that $S^a=S+aL$ so the quotient $\frac{S^a-S}{a}$ here is equal to $L$ for any $a$. Then, by Theorem \ref{teo:nl}, we have uniqueness of the location. Observe also that the conclusion follows directly from the fact that the paths do not have negative jumps: for $c>0$ and $a$ small, the location of the supremum of $X^a$ is always at the point $1$, and for $c<0$ and $a$ small, the maximizer is located at the random point $L$.\\

In the case where $c=0$, $S^a=S+a$ for $a>0$ and $S^a=S+aL$ for $a<0$. Then
$$ \frac{S^a-S}{a}= 1_{(0,\infty)} (a) + L 1_{(-\infty, 0)}(a) .$$
Using Theorem \ref{teo:nl} we conclude the equality between the events $A$ and $\{ L =1 \}$. Define the reversed process $\tilde X$ by
$$ \tilde X(s) := X(1-s)^{-} - X(1)    , \quad \forall 0 \leq s \leq 1 .$$
By the assumption $\sigma = c =0$ we have that 
$$ L =  \inf \{ s > 0:  \tilde X(s) \neq 0  \}  ,$$
 and since the reversed process $\tilde X$ has the same law as $X$ (see Lemma 3.4, p. 77 in \cite{Kyp}) we conclude the last part of the result.
\end{proof}

\subsection{Gaussian processes}

Before proving an analogue to Lemma $2$ in \cite{Pi}, we state some results for bridges of gaussian processes, known as anticipative representations, which are based on Karhunen-Lo\`eve orthogonal decompositions for gaussian processes. See Proposition 4 in \cite{GasSo} and Remark 3.6 (ii) in \cite{SoYa}.

\begin{lem}\label{lemLaxM}
Let $\{X(z) : z \in K \}$ be a centered gaussian process defined on a compact $K \subseteq \mathbb R^d$ containing the origin such that $X(\0)=0$, with covariance function $R$. Fix a set of different values $\{t^1,...,t^n\} \subseteq K$ such that $R(t^i,t^j)>0$ for all $i,j$. Define the process $X^{t^1,...,t^n }(z)$ as the process $X$ conditioned on the event $\{X(t^1)=0,...,X(t^n)=0 \}$. For $1 \leq k \leq n$ we construct the processes $\{ X^k(t) \}$ recursively:
\begin{eqnarray*}
X^0(z) &:=&  X(z)  ,\\
X^k(z) &:= & X^{k-1} (z) -  \frac{ R_{k-1} (z, t^k) }{R_{k-1} (t^k,t^k) } X^{k-1} (t^k)      \qquad \forall \, 1 \leq k \leq n,
\end{eqnarray*}
where the functions $R_k$ are also recursively constructed:
\begin{eqnarray*}
R_0(u,v) &:=& R(u,v), \\
R_k(u,v) & := & R_{k-1}(u,v) - \frac{ R_{k-1} (u,t^k) R_{k-1} (v,t^k)} { R_{k-1}(t^k,t^k)}    \qquad \forall \, 1 \leq k \leq n.
\end{eqnarray*} 
Then $X^k$ has the same law as $X^{t^1,...,t^k}$  for all $0 \leq k \leq n$.
\end{lem}

Based on this representation, we are able to express the distribution of a gaussian process as the sum of a gaussian bridge plus some independent normal variables weighted by some deterministic functions.

\begin{prop}\label{propgauss}
Let $K \subseteq \mathbb {R}^d$ be a non-empty compact set containing the origin and a centered Gaussian process $X$ defined on $K$, with covariance function $R$. Let $t^1,...,t^d$ be values in $K$ such that the vector $(X(t^1),...,X(t^d))$ has an invertible diagonal covariance matrix $\Sigma$. Then there exist functions $\gamma^i : K \rightarrow \mathbb R$ (which depend on $t^1,...,t^d$) which satisfy $\gamma^i(t^j) = \delta_{i,j}$ for all $ 1 \leq i, j \leq d$. Also, those functions are such that the next equality in distribution of processes holds:
$$ X(z) \eqdi X^{t^1,...,t^d} (z) + \sum_{i=1}^{d} \gamma^i (z) N_i ,$$
where $X^{t^1,...,t^d}$ has the distribution of $X$ conditioned on the event $\{X(t^1)=0,...,X(t^d)=0 \}$ and, for $ 1 \leq i \leq d$, $N_i$ is a centered normal variable with variance $R(t^i,t^i)$, where all random elements are independent.
\end{prop}

The proof of Proposition \ref{propgauss} is straightforward, based on Lemma \ref{lemLaxM}. The functions
\begin{equation}\label{def:gamma}
\gamma^i (z):= \frac{R(z,t^i)}{R(t^i,t^i)}  \qquad \forall 1 \leq i \leq d,
\end{equation}
satisfy the conditions stated in Proposition \ref{propgauss}. The details are presented in the Appendix.  \\

We are in conditions to present a generalization of Lemma $2$ in \cite{Pi} for Gaussian processes.

\begin{lem}\label{lem1gau}
Let $K \subseteq \mathbb {R}^d$ be a non-empty compact set containing the origin and $\{X(z) : z \in K \}$ be a centered Gaussian process, with covariance function $R$. Let $Y=Y(X)$ be some measurable functional of $X$ on $K$ such that $\E(Y^2) < \infty$ and fix some values $t^1,...,t^d$ in $K$ such that 
\begin{enumerate}
\item the gaussian vector $(X(t^1+a_1 \ee^1),...,X(t^d+a_d \ee^d))$ has invertible covariance matrix $\Sigma_a$, for all $a=(a_1,...,a_d)$ in a vicinity around zero (here $\{ \ee^i\}$ is the canonical base), and
\item the vector $(X(t^1),...,X(t^d))$ has a diagonal invertible covariance matrix $\Sigma_\0=\{ \sigma_{i,j} \}$.
\end{enumerate}
For those vectors $a$'s, define
$$ X^a(z) :=  X(z) + \sum_{i=1}^{d} a_i \gamma^i(z),  \quad \forall z \in K,$$
where the $\gamma^i(z)$ are defined by \eqref{def:gamma}, $Y^a:=Y(X^a)$ and $y(a):= \E(Y^a)$.
Then 
$$ \nabla y(\0) = \Big( \frac{\Cov(Y, X(t^1))}{\sigma_{1,1} }  , ... , \frac{\Cov(Y, X(t^d))}{\sigma_{d,d} }  \Big). $$
\end{lem}

\begin{proof}

We have that $X$ is equal in distribution as a process to
$$ X(z) \eqdi X^{t^1, ..., t^d} (z) + \sum_{i=1}^{d} \gamma^i (z) N_i ,$$
with all the elements defined as in Proposition \ref{propgauss}.

Then
\begin{equation}\label{gaussiansum}
 X^a(z) \eqdi X^{t^1,...,t^d} (z) + \sum_{i=1}^{d} \gamma^i(z) N_i  + \sum_{i=1}^{d} a_i \gamma^i(z) .
\end{equation}
Let $u=(u_1,...,u_d)$ be some element in $\mathbb R ^d$. Conditioned on the event $\{X^a(t^1)=u_1, ...,  X^a(t^d) =u_d\}$ it must happen that $a_1 + N_1 = u_1,...,a_d+N_d=u_d$ (evaluate \eqref{gaussiansum} on the points $t^1,...,t^d)$. It follows that
\begin{equation}\label{nondep}
 X^u(z) \eqdi X^{t^1,...,t^n} (z) + \sum_{i=1}^{d} u_i \gamma^i(z),
\end{equation}
on that event. \\

Define the density $K_u(a)$ by
$$K_u(a) := \frac{1}{(2 \pi)^{ \frac{d}{2} }} | \Sigma_a | ^{-\frac{1}{2}} \exp \Big\{ -\frac{1}{2} (u-a)' \Sigma_a^{-1}(u-a) \Big\} ,$$
where $\Sigma_a$ is the covariance matrix of the vector $(X^a(t^1),..., X^a(t^d))$.
By noting that 
$$(u-a)' \Sigma_a^{-1}(u-a) = \sum_{i=1}^{d} \sum_{j=1}^{d} (a_i-u_i) (a_j-u_j) \tilde \sigma^a_{i,j} ,$$
 where $\tilde \sigma^a_{i,j}$ is the $(i,j)$-entry of $\Sigma_a^{-1}$, we can compute the partial derivative of $K_u(a)$ respect to $a_k$, $1\leq k \leq d$:
$$ \frac{ \partial K_u(a)}{ \partial a_k} = K_u(a) \Big( -\frac{1}{2}  \sum_{i=1}^{d} \sum_{j=1}^{d} \tilde \sigma^a_{i,j} [ (a_i-u_i) \delta_{k,j} + (a_j-u_j) \delta_{i,k} ] \Big) = K_u(a) \Big( \sum_{i=1}^{d} \tilde \sigma^a_{i,k} (u_i -a_i) \Big) ,$$
where we used that $\Sigma_a$ is symmetric (so it is $\Sigma_a^{-1}$). 

We can express 
$$y(a)= \int \cdots \int \E (Y^a | X^a(t^1)=u_1, ..., X^a(t^d)=u_d  ) K_u(a) du_1 ... du_d$$ 
so, by interchanging derivative and integral operations,
\begin{eqnarray*}
 \frac{ \partial y(a)}{ \partial a_k} & =& \int \cdots \int  \frac{ \partial}{ \partial a_k} [ \E (Y^a | X^a(t^1)=u_1, ..., X^a(t^d)=u_d ) K_u(a) ] du_1 ... du_d \\
&=& \int \cdots \int \E (Y^a | X^a(t^1)=u_1, ..., X^a(t^n)=u_d ) \frac{ \partial K_u(a) }{ \partial a_k} du_1 ...du_d,
\end{eqnarray*}
where last equality is due to the non-dependence of $\E (Y^a | X^a(t^1)=u_1,..., X^a(t^d)=u_d)$ on the parameter $a$ given $(u_1,...,u_d)$, expressed in \eqref{nondep}. By substituting the derivative of $K_u(a)$ we obtain
$$ \frac{ \partial y(a)}{ \partial a_k} = \sum_{i=1}^{d} \tilde \sigma^a_{i,k} ( \E(Y^a X^a(t^i)) - a_i \E(Y^a)),$$
and the result follows by putting $a=\0$.
\end{proof}

\textbf{Proof of Theorem \ref{theo:gauss}}:

\begin{proof} 

By hypothesis, the  gaussian vector $(X(t^1),...,X(t^d))$ has an invertible diagonal covariance matrix given by 
$$\Sigma:= \{ \Cov( X(t^i), X(t^j) ) \}_{i,j} = \{ \sigma_{i,i} \delta_{i,j} \}_{i,j},$$
for some positive constants $\{\sigma_{i,i} \}$. For $a$ in $\mathbb{R}^d$ define the process $X^a$ by
\begin{equation}\label{eq:pert}
X^a(z) :=  X(z) + \sum_{i=1}^{d} a_i \frac{ R(z, t^i  )}{R(t^i,t^i) } \qquad \forall \, z \in K.
\end{equation}
 The functions $\Big\{  \frac{ R(z, t^i  )}{R(t^i,t^i) }   \Big\}_{i=1}^{d}$ satisfy the conditions stated in Proposition \ref{propgauss}. Therefore, if we denote the maximum of $X^a$ by $s(a)$, Lemma \ref{lem1gau} implies that 
$$ \nabla s (\0) = \Big( \frac{\partial s(a) }{ \partial a_i} \Big|_{a_i=0}  \Big)_{i=1}^d = \Big( \frac{1}{\sigma_{1,1}} \Cov(M,X(t^1), ..., \frac{1}{\sigma_{d,d}} \Cov(M,X(t^d) ) \Big) .$$
On the other hand, if we define $u(a)= \Big( \frac{a_1}{\sigma_{1,1}}, ..., \frac{a_d}{\sigma_{d,d}} \Big)$, \eqref{eq:pert} can be rewritten as
$$ X^u(z)= X(z) + \sum_{u=1}^{d} u_i R(z,t^i).$$
Since 
$$ \frac{ \partial s(a) }{ \partial a_i} = \frac{ \partial s(u) }{ \partial u_i} \cdot \frac{ \partial u_i }{ \partial a_i} = \frac{1}{\sigma_{i,i}} \frac{ \partial s(u) }{ \partial u_i}, $$
and the functions $R(z,t^i)$ meet the hypothesis of Theorem \ref{teo:multnl},  the location of the maximum is $a.s.$ unique and 
we conclude that
$$ \esp(Z) = \Big( \Cov(S, X(t^i) ) \Big)_{i=1}^d.$$
\end{proof}

\textbf{Proof of Theorem \ref{theo:gauss1}:}
\begin{proof}
It can be deduced from Theorem \ref{teo:nl}, as Theorem \ref{theo:gauss} was deduced from \ref{teo:multnl}. Alternatively, it is a particular case of Theorem \ref{theo:gauss}, taking $d=1$.
\end{proof}

\appendix
\section{}\label{app}

\textbf{Proof of Proposition \ref{propgauss}:} 
\begin{proof}
	
Define a process $\tilde X$ by 
$$ \tilde X (z) = X^{t^1,...,t^{d}} (z) + \sum_{i=1}^{d} \gamma^i(z) N_i \qquad \forall \, z \in K,    $$
such that all elements satisfy the hypothesis stated in the Proposition \ref{propgauss}, with the functions $\gamma^i$ need to be determined. Note that, by the independence of $X^{t^1,...,t^{k}}$ and $ \{ N_i\}$, $\tilde X$ is a centered gaussian process and its law is completely characterized by its covariances. Then, $\tilde X$ has the same distribution as $X$ if and only if $R(u,v)= \Cov( \tilde X(u), \tilde X(v))$ for all $u,v$ in $K$. Under all the hypothesis, that condition can be rewritten as
\begin{equation}\label{covgamma}
R(z,v) = R_d(z,v) + \sum_{j=1}^{d} \gamma^j(z) \gamma^j(v) \Var(N_j)  \qquad \forall z,v \in K,
\end{equation}
where $R_d$ is defined as in Lemma \ref{lemLaxM} (note that $\Cov(X^{t^1,...,t^{d}} (z), X^{t^1,...,t^{d}} (v))=R_d(z,v)$). Put $v=t^i$ in \eqref{covgamma}; by the assumption $\gamma^j(t^i)= \delta_{i,j}$, it reduces to
$$ R(z,t^i) = R_d(z,t^i) + \gamma^i(z) \Var(N_i) = \gamma^i(z) \Var(N_i),$$
where last equality is due to the fact that $X^{t^1,...,t^d}(t^i)=0$. Since putting $z=t^i$ in the last equation leads to $\Var(N_i)=R(t^i,t^i)$, we conclude that \eqref{def:gamma} is indeed the solution of \eqref{covgamma}, so the equality in distribution is proved. Since the covariance matrix of $(X(t^1,...,X(t^d))$ is diagonal we have that $\gamma^i(t^j)= \frac{R(t^j,t^i)}{R(t^i,t^i)}= \delta_{i,j} $.
\end{proof}

\end{document}